\documentclass[11pt,a4paper]{article}

\usepackage{amsmath,amsthm}

\PassOptionsToPackage{backend=biber}{biblatex}
\usepackage{cgeometry}
\allowdisplaybreaks
\newcommand{\e}{\mathrm e}

\title{A counterexample to the zero-mass conjecture}
\author{Long Li \and Mingchen Xia}
\date{}

\begin{document}

\maketitle

\begin{abstract}
We construct an explicit negative plurisubharmonic function on the unit polydisc in $\mathbb C^n, n\ge 2$ with an isolated singularity at the origin. Its Monge--Amp\`{e}re measure is the Dirac mass at the origin, but its Lelong number vanishes. This gives a negative answer to the zero-mass conjecture of Guedj--Rashkovskii in every complex dimension at least two.
\end{abstract}

\section{Introduction}

Let $u$ be plurisubharmonic near $0\in\mathbb C^n$ and locally bounded on a punctured neighborhood of $0$. For such an isolated singularity, Demailly's construction always defines its Monge--Amp\`ere measure
\[
  \MA(u)=(\ddc u)^n
\]
as a positive Radon measure; see \cite[Chapter~III, Proposition~4.1 and Corollary~4.3]{DemaillyCADG}. Its residual Monge--Amp\`ere mass at the origin is
\[
  \tau(u,0)\coloneqq\MA(u)(\{0\}).
\]
The comparison theorem gives the fundamental lower bound $\nu(u,0)^n\le\tau(u,0)$; see, for example, \cite{Demailly1993}. Thus a positive Lelong number forces a positive residual mass. The converse question,
\[
  \nu(u,0)=0\quad\Longrightarrow\quad\tau(u,0)=0,
\]
is the Guedj--Rashkovskii zero-mass conjecture. It was already considered in \cite{Gue10} and \cite{Ras01}. It appears as Question~7 in \cite{DGZ15} and is discussed in detail in \cite{Ras16}.

In the last three decades, there have been many important contributions related to this conjecture; see, for example, \cite{BFJ08,KR21,Ras01,Ras06,Ras13,Ras16,Wik05}.
Several important special cases are known. The conjecture holds for isolated singularities with finite {\L}ojasiewicz exponent; in particular, it holds for isolated analytic and multicircled singularities \cite{Ras16}. Further sufficient conditions in the setting of Cegrell energy classes were obtained by {\AA}hag, Cegrell, and Ph\d{a}m \cite{ACP19}.

Another line of work exploits symmetry. The first-named author proved the conjecture for $S^1$-invariant isolated singularities in $\mathbb C^2$ \cite{Li24}, and He, Xu and the first-named author subsequently proved the corresponding result in arbitrary complex dimension \cite{HLX23}. In complex dimension two, the same authors established a residual-mass upper bound, and hence the zero-mass conclusion, under uniform directional Lipschitz continuity. They also showed that separation of the circular direction in the alternating part implies this regularity \cite{HLX25}.

More recently, Deng, Li, Liu, Wang and Zhou introduced the log truncated threshold and proved an optimal residual-mass estimate in terms of higher Lelong numbers when this threshold is finite. They also established a generic zero-mass theorem for parametrized maxima in Cegrell energy classes, with a pluripolar exceptional set \cite{DLLWZ26}.

The result below shows that without such additional structure the proposed implication fails already in complex dimension two. In fact, the entire Monge--Amp\`ere mass can remain concentrated at a pole whose Lelong number vanishes. For $N\ge1$ and $\zeta=(\zeta_1,\ldots,\zeta_N)\in\mathbb C^N$, write
\[
  \|\zeta\|_\infty=\max_{1\le\ell\le N}|\zeta_\ell|,
  \qquad
  \mathbb D^N=\{\zeta\in\mathbb C^N:\|\zeta\|_\infty<1\}.
\]
\begin{theorem}\label{thm:main}
There exists an explicitly defined negative function $u\in\PSH(\mathbb D^2)$ such that $\e^u$ is continuous on $\mathbb D^2$ and
\begin{equation}\label{eq:main}
  u^{-1}(-\infty)=\{0\},
  \qquad
  \nu(u,0)=0,
  \qquad
  \MA(u)=\delta_0.
\end{equation}
\end{theorem}

Thus $u$ is continuous as an extended-real-valued function; in particular, it is continuous outside its isolated pole. The function is obtained by an explicit recursive construction. The results cited above imply that $u$ is not $S^1$-invariant and does not satisfy the uniform directional Lipschitz condition of \cite{HLX25}. Moreover, if its log truncated threshold were finite, the residual-mass estimate of \cite[Theorem~1.3]{DLLWZ26} would imply $\tau(u,0)=0$, since $\nu(u,0)=0$. Hence its log truncated threshold is infinite.

The two-dimensional example also gives counterexamples in every higher dimension.

\begin{corollary}\label{cor:higher-dimensional}
For every $n\ge2$, there exists a negative function $u_n\in\PSH(\mathbb D^n)$ such that $\e^{u_n}$ is continuous on $\mathbb D^n$ and
\[
  u_n^{-1}(-\infty)=\{0\},
  \qquad
  \nu(u_n,0)=0,
  \qquad
  \MA(u_n)=\delta_0.
\]
\end{corollary}

The construction is based on the polynomial maps
\[
  F_\varepsilon(z,w)=\left(\frac{1}{2}w^2,\frac{1}{2}z^2+\varepsilon w\right).
\]
We choose a rapidly decreasing sequence $\varepsilon_j>0$ and compose these maps to form $H_j=F_{\varepsilon_j}\circ\cdots\circ F_{\varepsilon_1}$. The quadratic terms make each $F_\varepsilon$ a finite map of degree $4$, and hence $H_j$ has degree $4^j$. The small linear term $\varepsilon w$ plays a different role: it preserves the direction $e_2=(0,1)$ under differentiation. Consequently, $H_j$ still has a nonzero linear term in that direction, even though its topological degree grows exponentially.

We turn the iterates into plurisubharmonic weights by setting $q_j=2^{-j}\log\|H_j\|_\infty$. Since $2^{-j}=(\deg H_j)^{-1/2}$, this normalization exactly compensates for the degree $4^j$ in complex dimension two and gives $\MA(q_j)=\delta_0$. On the other hand, the surviving linear term makes the minimum of the vanishing orders of the components of $H_j$ equal to one, and therefore $\nu(q_j,0)=2^{-j}\to0$. Thus the iterates separate Monge--Amp\`ere mass from Lelong number at each finite stage, but the functions $q_j$ need not form a decreasing sequence.

To produce a limit, we truncate at increasingly deep levels and set $V_j=\max\{q_j,-4^j\}$. The rapid decay of $\varepsilon_j$ is chosen so that $V_{j+1}\le V_j$, while the cutoff cores $\{q_j\le-4^j\}$ shrink to the origin. Demailly's sweeping formula shows that truncation does not remove the unit Monge--Amp\`ere mass: it moves that mass from the pole to the cutoff level $\{q_j=-4^j\}$. Since these level sets collapse to the origin, $\MA(V_j)\rightharpoonup\delta_0$. Monotone continuity then gives $\MA(u)=\delta_0$ for $u=\lim_{j\to\infty}V_j$.

It remains to prove that the limiting singularity has zero Lelong number, since Lelong numbers need not pass continuously to a decreasing limit. A direct lower bound along the positive real ray in the direction $e_2$ produces points $\zeta_j\to0$ at which every later cutoff is inactive and
\[
  \frac{u(\zeta_j)}{\log\|\zeta_j\|_\infty}\longrightarrow0.
\]
The liminf characterization of the Lelong number then gives $\nu(u,0)=0$.

\Cref{sec:demailly} records the required facts about plurisubharmonic functions with isolated poles, and \cref{sec:ex} gives the construction and proves \cref{thm:main} and \cref{cor:higher-dimensional}.

\subsection*{Acknowledgments}

An initial idea was constructed by the \href{https://github.com/frenzymath/Rethlas}{Rethlas agent} using the \texttt{gpt-5.6-sol} model. That construction involved complicated polynomial iterations. Guided by its underlying principles, the authors subsequently found the simpler construction presented here.

The authors thank Xiuxiong Chen for his continuous support and encouragement. They are especially grateful to Per~{\AA}hag for his careful reading of the manuscript and for suggesting the streamlined argument in \cref{sec:lelong}.

The first-named author is supported by the Open Foundation of the State Key Laboratory of Mathematical Sciences (Grant No.~SLMS 2025 KFKT TD 12).
The second-named author is supported by the National Key R\&D Program of China under Grant No.~2025YFA1018200.

\section{Preliminaries}\label{sec:demailly}

We collect three tools used in the construction: monotone continuity of the Monge--Amp\`ere operator for isolated poles, Demailly's sweeping formula, and the mass formula for logarithms of finite holomorphic maps. We also recall the characterization of Lelong numbers needed for the witness argument.

We use
\begin{equation}\label{eq:normalization}
  \mathrm{d}^{\mathrm{c}}=\frac{1}{2\pi\mathrm{i}}(\partial-\bar\partial),
  \qquad
  \ddc=\frac{\mathrm{i}}{\pi}\partial\bar\partial,
\end{equation}
so that $\ddc\log|z|=\delta_0$ in one complex dimension.

Throughout, weak convergence refers to Radon measures. For Radon measures on a domain $Y$, we write $\mu_k\rightharpoonup\mu$ if
\[
  \int_Y\chi\,\mathrm d\mu_k\longrightarrow
  \int_Y\chi\,\mathrm d\mu
  \qquad\text{for every }\chi\in C_c(Y).
\]
This is the meaning of $\rightharpoonup$ throughout the paper; positive top-degree currents are identified with their associated Radon measures.

Following \cite[Definition~1.1]{DGZ15}, we say that a plurisubharmonic function $w$ on a domain $X\subset\mathbb C^2$ has a \emph{well-defined Monge--Amp\`ere measure} if there is a Radon measure $\mu$ on $X$ such that, for every domain $V\Subset X$ and every sequence $w_k\in\PSH(V)\cap C^\infty(V)$ decreasing pointwise to $w|_V$, the measures $(\ddc w_k)^2\rightharpoonup\mu|_V$ on $V$. When this condition holds, the measure $\mu$ is uniquely determined, and we write $\MA(w)=\mu$.

We record a few standard facts in the normalization \eqref{eq:normalization}.
\begin{proposition}\label{prop:question7-product}
Let $X\subset\mathbb C^2$ be a domain and let $w\in\PSH(X)$ have discrete unbounded locus. Then $w\,\ddc w$ has locally finite mass, and
\[
  \mu_w\coloneqq\ddc\bigl(w\,\ddc w\bigr)
\]
is a positive Radon measure compatible with restriction. If $w_k\in\PSH(X)$ decreases pointwise to $w$, then
\[
  \mu_{w_k}\rightharpoonup\mu_w.
\]
For locally bounded $w_k$, the measure $\mu_{w_k}$ is the usual Bedford--Taylor product $(\ddc w_k)^2$ \cite{BT82}.

Consequently, $w$ has a well-defined Monge--Amp\`ere measure in the sense used above, and $\MA(w)=\mu_w$. For $c>0$,
\begin{equation}\label{eq:ma-homogeneity}
  \MA(cw)=c^2\MA(w).
\end{equation}
\end{proposition}

\begin{proof}
The finite-mass statement is \cite[Chapter~III, Proposition~4.1]{DemaillyCADG}, applied with $T=\ddc w$; the Stein boundary-cover hypothesis there is automatic for a discrete unbounded locus. The convergence is the case $p=q=2$ and $T=1$ of \cite[Chapter~III, Corollary~4.3]{DemaillyCADG}. Notice that $w_k\ge w$, so the unbounded locus of $w_k$ is contained in that of $w$. Positivity, restriction compatibility, and the identity in \eqref{eq:ma-homogeneity} follow from the same construction.
\end{proof}

\begin{lemma}\label{lem:sweeping}
Let $p\ge1$, let $X\subset\mathbb C^p$ be a pseudoconvex domain containing $0$, and let $\varphi\in\PSH(X)$ have discrete unbounded locus. Suppose that $\e^\varphi$ is continuous, and 
\[
  (\ddc\varphi)^p=m\delta_0 \quad\textup{for some }m\ge0,\qquad \varphi(0)=-\infty,
\]
with $\{\varphi<R\}\Subset X$ for some $R\in\mathbb R$. If $r<R$, then $(\ddc\max\{\varphi,r\})^p$ is supported on $\{\varphi=r\}$ and has total mass $m$.
\end{lemma}

\begin{proof}
Demailly's swept-out measure at level $r$ is
\[
  \left(\ddc\max\{\varphi,r\}\right)^p-\mathbf 1_{\{\varphi\ge r\}}(\ddc\varphi)^p.
\]
It is positive, supported on $\{\varphi=r\}$, and has mass $\int_{\{\varphi<r\}}(\ddc\varphi)^p$ \cite[Chapter~III, formulas~(6.1)--(6.2)]{DemaillyCADG}. Here the restriction term vanishes and the last integral equals $m$, because the entire measure $m\delta_0$ lies in $\{\varphi<r\}$.
\end{proof}

\begin{lemma}\label{lem:finite-mass}
Let $G\colon U\to W$ be a finite holomorphic map of degree $d$ between domains in $\mathbb C^2$, each containing $0$, and suppose that $G^{-1}(0)=\{0\}$. Then
\begin{equation}\label{eq:finite-mass}
  \MA\bigl(\log\|G\|_\infty\bigr)=d\delta_0
\end{equation}
on $U$.
\end{lemma}

\begin{proof}
Demailly's max-norm computation gives $\MA(\log\|\cdot\|_\infty)=\delta_0$ \cite[Chapter~III, Corollary~7.4]{DemaillyCADG}. For a smooth plurisubharmonic weight $\psi$ on $W$, change of variables for a finite map, with local multiplicities counted, gives
\[
  G_*\MA(\psi\circ G)=d\,\MA(\psi).
\]
Here the push-forward is defined by $(G_*\mu)(E)=\mu(G^{-1}(E))$ for Borel sets $E\subset W$.
Local regularization and Bedford--Taylor continuity extend this identity to every locally bounded plurisubharmonic weight $\psi$ on $W$ \cite[Chapter~III, Theorem~3.7]{DemaillyCADG}.

Put $\phi=\log\|\cdot\|_\infty$ and apply the identity to the locally bounded weights $\psi_T=\max\{\phi,-T\}$. Since $\psi_T\downarrow\phi$, \cref{prop:question7-product} gives
\[
  \MA(\psi_T)\rightharpoonup\delta_0, \qquad \MA(\psi_T\circ G)\rightharpoonup\MA(\phi\circ G)
\]
as $T\to\infty$. Passing to the limit in the push-forward identity gives
\[
  G_*\MA(\phi\circ G)=d\delta_0.
\]
The measure on the left is positive, and $G^{-1}(0)=\{0\}$; hence $\MA(\phi\circ G)$ is supported at $0$ and has mass $d$. This is \eqref{eq:finite-mass}.
\end{proof}

\begin{lemma}\label{lem:demailly-lelong}
Let $v$ be plurisubharmonic near $0\in\mathbb C^2$. Then
\begin{equation}\label{eq:lelong-liminf}
  \nu(v,0)=\liminf_{\zeta\to 0} \frac{v(\zeta)}{\log\|\zeta\|_\infty}.
\end{equation}
Moreover, if $f$ is a nonzero holomorphic germ, then
\begin{equation}\label{eq:lelong-order}
  \nu(\log|f|,0)=\operatorname{ord}_0 f.
\end{equation}
Consequently, if $F=(f_1,\ldots,f_N)$ is a holomorphic germ such that $f_\ell\not\equiv0$ and $f_\ell(0)=0$ for every $\ell$, then
\begin{equation}\label{eq:lelong-holomorphic-map}
  \nu(\log\|F\|_\infty,0)=\min_{1\le\ell\le N}\operatorname{ord}_0 f_\ell.
\end{equation}
Also $\nu(cv,0)=c\nu(v,0)$ for every $c>0$.
\end{lemma}

\begin{proof}
Formula \eqref{eq:lelong-liminf} follows from \cite[Chapter~III, formulas~(6.9c)--(6.9e) and Corollary~7.3]{DemaillyCADG}, and \eqref{eq:lelong-order} is the Lelong--Poincar\'e special case recorded in \cite[Chapter~III, \S6.10]{DemaillyCADG}. For \eqref{eq:lelong-holomorphic-map}, let $m=\min_\ell\operatorname{ord}_0 f_\ell$ and choose $\ell_0$ with $\operatorname{ord}_0 f_{\ell_0}=m$. Near the origin,
\[
  \log|f_{\ell_0}| \le \log\|F\|_\infty \le m\log\|\zeta\|_\infty+O(1),
\]
and the two preceding formulas give the claim. The homogeneity is immediate from \eqref{eq:lelong-liminf}.
\end{proof}

\section{The example}\label{sec:ex}

\subsection{Finite-stage potentials}

For $0<\varepsilon\le\frac{1}{4}$, put
\[
  F_\varepsilon(z,w)
  =\left(\frac{1}{2}w^2,\frac{1}{2}z^2+\varepsilon w\right).
\]
If $r=\|(z,w)\|_\infty\le1$, then
\[
  \|F_\varepsilon(z,w)\|_\infty \le\frac{1}{2}r^2+\frac{1}{4}r \le\frac{3}{4}<1.
\]
Thus $F_\varepsilon(\overline{\mathbb D^2})\subset\mathbb D^2$.
Set $e_2=(0,1)$, the direction in which a linear term will survive.

For each $j\ge1$, set
\[
  \varepsilon_j=\frac{1}{4}\e^{-8^{j-1}},
  \qquad
  c_j=\prod_{k=1}^j\varepsilon_k=4^{-j}\e^{-\frac{8^j-1}{7}}.
\]

\begin{lemma}\label{lem:algebra}
The map $F_\varepsilon\colon\mathbb C^2\to\mathbb C^2$ is finite of degree $4$, with zero fiber $\{0\}$. If
\[
  F_j=F_{\varepsilon_j}, \qquad H_j=F_j\circ\cdots\circ F_1,
\]
then
\begin{equation}\label{eq:degree}
  H_j^{-1}(0)=\{0\},
  \qquad
  \deg H_j=4^j.
\end{equation}
Moreover, the local degree of $H_j$ at the origin is $4^j$, and
\begin{equation}\label{eq:linear}
  \mathrm dH_j(0)e_2=c_j e_2.
\end{equation}
For every $t\in(0,1)$,
\begin{equation}\label{eq:positive-axis}
  \left\|H_j(t e_2)\right\|_\infty\ge c_j t.
\end{equation}
\end{lemma}

\begin{proof}
Write
\[
  F_\varepsilon(z,w)=(\alpha,\beta).
\]
Then
\begin{equation}\label{eq:finite-fiber}
  w^2=2\alpha,\qquad z^2=2\beta-2\varepsilon w.
\end{equation}
Thus $w$ is integral over $\mathbb C[\alpha,\beta]$, while $z$ is integral over $\mathbb C[\alpha,\beta,w]$. By transitivity of integrality, $\mathbb C[z,w]$ is integral, and hence finite, over $\mathbb C[\alpha,\beta]$. Therefore $F_\varepsilon$ is a finite polynomial map. Equation \eqref{eq:finite-fiber} gives at most four points in each fiber and exactly four in a generic fiber, so the degree is four; it also shows that the zero fiber is $\{0\}$.

Finiteness is preserved under composition, degrees multiply, and induction gives $H_j^{-1}(0)=\{0\}$. For a finite map, the local multiplicities in a fiber sum to the global degree. Since $H_j^{-1}(0)=\{0\}$, the local degree of $H_j$ at the origin equals its global degree. This proves \eqref{eq:degree} and the local-degree assertion.

At the origin,
\[
  \mathrm dF_\varepsilon(0)= \begin{pmatrix}0&0\\0&\varepsilon\end{pmatrix}.
\]
The chain rule gives \eqref{eq:linear}. To prove \eqref{eq:positive-axis}, write
\[
  H_j(t e_2)=\Bigl(z_j(t),w_j(t)\Bigr),\qquad \Bigl(z_0(t),w_0(t)\Bigr)=(0,t).
\]
The recursion is
\[
  z_j(t)=\frac12 w_{j-1}(t)^2,
  \qquad
  w_j(t)=\frac12 z_{j-1}(t)^2+\varepsilon_j w_{j-1}(t).
\]
Induction shows that $z_k(t)$ and $w_k(t)$ are nonnegative real numbers for every $k\ge0$. Consequently,
\[
  w_j(t)\ge\varepsilon_j w_{j-1}(t)
  \ge\cdots\ge c_j t,
\]
which proves \eqref{eq:positive-axis}.
\end{proof}

On $\mathbb C^2$, define
\[
  q_j(\zeta)=2^{-j}\log\|H_j(\zeta)\|_\infty,
\]
with value $-\infty$ at $0$. We use the same symbol for its restriction to $\mathbb D^2$.

\begin{proposition}
Each $q_j$ is negative and plurisubharmonic on $\mathbb D^2$, continuous on $\mathbb D^2\setminus\{0\}$, and tends to $-\infty$ at $0$. Moreover,
\begin{equation}\label{eq:finite-invariants}
  \nu(q_j,0)=2^{-j}, \qquad \MA(q_j)=\delta_0.
\end{equation}
\end{proposition}

\begin{proof}
The logarithm of the maximum of finitely many moduli of holomorphic functions is plurisubharmonic, so $q_j\in\PSH(\mathbb D^2)$. Since $H_j(\mathbb D^2)\subset\mathbb D^2$, we have $q_j<0$. Moreover, $H_j^{-1}(0)=\{0\}$ by \eqref{eq:degree}; hence $q_j$ is continuous on $\mathbb D^2\setminus\{0\}$ and tends to $-\infty$ at $0$.

Because $H_j$ is finite, neither of its components is identically zero. Both components vanish at $0$, while \eqref{eq:linear} gives $\mathrm dH_j(0)e_2=c_j e_2\ne0$. Thus the minimum of their vanishing orders is one, so $\nu(q_j,0)=2^{-j}$ by \cref{lem:demailly-lelong}. By \cref{lem:algebra,lem:finite-mass} and \eqref{eq:ma-homogeneity},
\[
  \MA(q_j)=2^{-2j}4^j\delta_0=\delta_0.
\]
\end{proof}

\subsection{Comparison and the decreasing limit}

The cutoff level $-4^j$ is matched to the next perturbation scale. Indeed,
\[
  q_j(\zeta)\ge-4^j \quad\Longleftrightarrow\quad\|H_j(\zeta)\|_\infty\ge\e^{-8^j}=4\varepsilon_{j+1}.
\]
We refer to this set as the visible region at stage $j$.

\begin{lemma}
For $\zeta\in\overline{\mathbb D^2}$ and $0<\varepsilon\le\frac{1}{4}$,
\begin{equation}\label{eq:inverse}
  \|F_\varepsilon(\zeta)\|_\infty \ge\frac{1}{8}\|\zeta\|_\infty^4.
\end{equation}
Consequently,
\begin{equation}\label{eq:q-lower}
  q_j(\zeta)\ge 2^j\log\|\zeta\|_\infty-(\log 2)(2^j-2^{-j}).
\end{equation}
\end{lemma}

\begin{proof}
Write $\zeta=(z,w)$, $r=\|\zeta\|_\infty$, and $\rho=\|F_\varepsilon(\zeta)\|_\infty$. Then
\[
\rho=\max\left\{ \frac{1}{2}|w|^2, \left|\frac{1}{2}z^2+\varepsilon w\right| \right\}.
\]
If $|w|\ge r^2/2$, then
\[
\rho \ge \frac{1}{2}|w|^2 \ge \frac{1}{2}\left(\frac{1}{2}r^2\right)^2=\frac{1}{8}r^4.
\]
Suppose now that $|w|<\frac{1}{2}r^2$. This case forces $r>0$, and since $r\le1$, we have $|w|<r$ and hence $r=|z|$. By the reverse triangle inequality and the assumption $\varepsilon\le\frac{1}{4}$, we obtain
\[
\left|\frac{1}{2}z^2+\varepsilon w\right|
\ge\frac{1}{2}r^2-\varepsilon|w|
>\frac{1}{2}r^2-\frac{1}{8}r^2
=\frac{3}{8}r^2
\ge\frac18r^4.
\]
Thus \eqref{eq:inverse} holds in both cases. Iterating it gives
\[
  \|H_j(\zeta)\|_\infty \ge 2^{1-4^j} \|\zeta\|_\infty^{4^j}.
\]
Taking logarithms and dividing by $2^j$ proves \eqref{eq:q-lower}.
\end{proof}

\begin{lemma}
For every $j\ge 1$,
\begin{equation}\label{eq:q-visible}
  -2^{-j}\log 2 \le q_{j+1}-q_j \le 0 \quad\textup{on }\{q_j\ge-4^j\},
\end{equation}
and
\begin{equation}\label{eq:q-global}
  q_{j+1}\le\max\{q_j,-4^j\} \quad\textup{on }\mathbb D^2.
\end{equation}
\end{lemma}

\begin{proof}
At $\zeta=0$, the global assertion is immediate and the visible-set assertion is vacuous. Assume henceforth that $\zeta\ne0$, write $H_j(\zeta)=(z,w)$, and put $r=\max\{|z|,|w|\}$. Then
\begin{equation}\label{eq:qjp1qj}
  q_{j+1}(\zeta)-q_j(\zeta)=2^{-j-1} \log\frac{\|F_{j+1}(H_j(\zeta))\|_\infty}{r^2}.
\end{equation}
If $q_j(\zeta)\ge-4^j$, then $r\ge \e^{-8^j}$. Since $\varepsilon_{j+1}=\e^{-8^j}/4\le r/4$, the first component of $F_{j+1}(z,w)$ has modulus at most $r^2/2$, while the second has modulus at most
\[
  \frac{1}{2}r^2+\varepsilon_{j+1}r
  \le\frac{3}{4}r^2.
\]
Therefore
\[
  \|F_{j+1}(z,w)\|_\infty\le r^2.
\]
If $r=|w|$, the first component is $r^2/2$; if $r=|z|$, then
\[
  \left|\frac{1}{2}z^2+\varepsilon_{j+1}w\right|
  \ge\frac{1}{2}r^2-\varepsilon_{j+1}r
  \ge\frac{1}{4}r^2.
\]
Thus the quotient in the logarithm in \eqref{eq:qjp1qj} lies in $[1/4,1]$, which gives \eqref{eq:q-visible} and the global upper bound \eqref{eq:q-global} on this set.

If $q_j(\zeta)<-4^j$, set $R_j=\e^{-8^j}$ and write $r=R_j x$ with $0<x<1$. Since $\varepsilon_{j+1}=R_j/4$,
\[
  \|F_{j+1}(z,w)\|_\infty
  \le R_j^2\left(\frac{x^2}{2}+\frac{x}{4}\right)
  <\frac{3}{4}R_j^2
  <R_j^2=\e^{-2\cdot8^j},
\]
and hence
\[
  q_{j+1}(\zeta)
  <-4^j.
\]
Together the two cases prove \eqref{eq:q-global}.
\end{proof}

Define
\begin{equation}\label{eq:Vj}
  V_j=\max\{q_j,-4^j\}
\end{equation}
on $\mathbb D^2$, and define its cutoff core by
\[
  K_j=\{\zeta\in\overline{\mathbb D^2}:q_j(\zeta)\le-4^j\}.
\]
Finally, define
\begin{equation}\label{eq:u}
  u=\inf_{j\ge1}V_j.
\end{equation}

\begin{proposition}\label{prop:limit}
Each $V_j$ is a bounded negative plurisubharmonic function, and the sequence decreases pointwise to $u$.
The limit is negative and plurisubharmonic on $\mathbb D^2$, continuous on $\mathbb D^2\setminus\{0\}$, and
\[
  u^{-1}(-\infty)=\{0\}.
\]
Moreover, $q_j\to u$ locally uniformly on $\mathbb D^2\setminus\{0\}$, and $\e^u$ is continuous on $\mathbb D^2$.
\end{proposition}

\begin{proof}
Each $V_j$ is bounded and plurisubharmonic. Since $H_j(\mathbb D^2)\subset\mathbb D^2$, both branches in \eqref{eq:Vj} are negative, so $V_j<0$. Moreover, \eqref{eq:q-global} gives $q_{j+1}\le V_j$, while
\[
  -4^{j+1}\le-4^j\le V_j.
\]
Taking the maximum of these two inequalities proves $V_{j+1}\le V_j$. Thus $V_j$ decreases pointwise to the function $u$ defined in \eqref{eq:u}.

If $\zeta\in K_j$, then \eqref{eq:q-lower} gives
\[
  2^j\log\|\zeta\|_\infty \le-4^j+(\log 2)\left(2^j-2^{-j}\right) \le-4^j+2^j\log 2,
\]
and therefore
\begin{equation}\label{eq:core}
  \|\zeta\|_\infty
  \le 2\e^{-2^j}.
\end{equation}

Moreover, $K_j=\{\zeta\in\overline{\mathbb D^2}:\e^{q_j(\zeta)}\le\e^{-4^j}\}$ is closed in $\overline{\mathbb D^2}$,
because $\e^{q_j}$ is continuous there. Thus the cores lie in closed max-norm balls whose radii tend to zero. Since $2\e^{-2}<1$, every $K_j$ is compactly contained in $\mathbb D^2$.

Let $L\Subset \mathbb D^2\setminus\{0\}$. By \eqref{eq:core}, there is an $N$ such that $L\cap K_k=\varnothing$ for every $k\ge N$. Hence \eqref{eq:q-visible} applies at each later stage, and for $\ell>k\ge N$,
\[
  0\le q_k-q_\ell
  \le\sum_{s=k}^{\ell-1}2^{-s}\log 2
  \le2^{1-k}\log 2
  \qquad\text{on }L.
\]
Thus $q_j$ is uniformly Cauchy on $L$. The cutoff in \eqref{eq:Vj} is also inactive on $L$ for $j\ge N$, so $V_j=q_j$ there. Therefore $V_j$ and $q_j$ have the same locally uniform limit on the punctured domain.

A decreasing limit of plurisubharmonic functions is plurisubharmonic unless it is identically $-\infty$. The punctured limit is finite, so that alternative is excluded, and $u\le V_1<0$. Finally,
\[
  V_j(0)=-4^j\longrightarrow-\infty.
\]
The locally uniform limit of the continuous functions $q_j$ is continuous away from $0$. To prove continuity at the pole, fix $A>0$ and choose $j$ with $4^j>A$. Since $\e^{q_j}$ is continuous and vanishes only at $0$, the core $K_j$ contains a neighborhood of $0$. On this neighborhood,
\[
  u\le V_j=-4^j<-A.
\]
Thus $u(\zeta)\to-\infty$ as $\zeta\to0$, so $\e^u$ is continuous at $0$ as well.
\end{proof}

\subsection{The atomic mass}

We now identify the limiting Monge--Amp\`ere measure.

\begin{proposition}\label{prop:mass}
The Monge--Amp\`ere measure of $u$ satisfies
\[
  \MA(u)=\delta_0.
\]
\end{proposition}

\begin{proof}
By \cref{prop:limit,prop:question7-product}, the Monge--Amp\`ere measure of $u$ is defined. The bound \eqref{eq:core} gives $K_j\Subset\mathbb D^2$.
The function $\e^{q_j}=\|H_j\|_\infty^{2^{-j}}$ is continuous on $\overline{\mathbb D^2}$ and vanishes only at $0$.
On the boundary $\partial \mathbb D^2$, equation \eqref{eq:q-lower} gives
\[
  q_j\ge-(\log 2)\left(2^j-2^{-j}\right)>-4^j.
\]
We may therefore choose
\[
  -4^j<s_j<\min_{\partial\mathbb D^2}q_j, \qquad \left\{q_j<s_j\right\}\Subset\mathbb D^2.
\]
Apply \cref{lem:sweeping} with $\varphi=q_j$, $r=-4^j$, and $R=s_j$. In view of \eqref{eq:finite-invariants}, $\MA(V_j)$ is a probability measure supported on $\{q_j=-4^j\}\subset K_j$.

For every $\chi\in C_c(\mathbb D^2)$,
\[
  \left|\int_{\mathbb D^2}\chi\,\MA(V_j)-\chi(0)\right|\le\sup_{K_j}\left|\chi-\chi(0)\right| \longrightarrow 0.
\]
Hence
\begin{equation}\label{eq:swept-limit}
  \MA(V_j)\rightharpoonup\delta_0.
\end{equation}
On the other hand, $V_j\downarrow u$ and the unbounded locus of $u$ is $\{0\}$. The monotone convergence in \cref{prop:question7-product} gives
\[
  \MA(V_j)\rightharpoonup\MA(u).
\]
Comparison with \eqref{eq:swept-limit} proves the proposition.
\end{proof}

\subsection{Evaluation of the Lelong number}\label{sec:lelong}

Lelong numbers need not pass continuously to decreasing limits, so we construct points at which every later truncation is inactive. We first record a consequence of the one-step estimate.

\begin{lemma}\label{lem:untruncated-region}
Let $j\ge1$ and $\zeta\in\mathbb D^2$.
If $q_j(\zeta)>-4^j$, then
\begin{equation}\label{eq:inactive-cutoffs}
  q_k(\zeta)>-4^k
  \qquad(k\ge j),
\end{equation}
and
\begin{equation}\label{eq:tail-bound}
  0\le q_j(\zeta)-u(\zeta)\le2^{1-j}\log 2.
\end{equation}
\end{lemma}

\begin{proof}
If $q_k(\zeta)>-4^k$, then \eqref{eq:q-visible} gives
\[
  q_{k+1}(\zeta)
  >-4^k-2^{-k}\log 2
  >-4^{k+1}.
\]
Induction proves \eqref{eq:inactive-cutoffs}. Thus none of the later truncations acts at $\zeta$. For $m>j$, successive applications of \eqref{eq:q-visible} give
\[
  0\le q_j(\zeta)-q_m(\zeta)
  \le\sum_{k=j}^{m-1}2^{-k}\log 2
  \le2^{1-j}\log 2.
\]
Letting $m\to\infty$ proves \eqref{eq:tail-bound}.
\end{proof}

\begin{proposition}\label{prop:lelong}
The function $u$ in \eqref{eq:u} satisfies
\[
  \nu(u,0)=0.
\]
\end{proposition}

\begin{proof}
For $j\ge2$, set
\[
  b_j=-q_j(c_j e_2).
\]
Since $0<c_j<1$, the point $c_j e_2$ belongs to $\mathbb D^2$. By \eqref{eq:positive-axis},
\[
  \|H_j(c_j e_2)\|_\infty\ge c_j^2.
\]
Consequently,
\[
  0\le b_j\le2^{1-j}(-\log c_j),
\]
where
\[
  -\log c_j=j\log 4+\frac{8^j-1}{7}.
\]
Hence
\[
  2^{1-j}\left(-\log c_j\right)=\left(\frac{2}{7}+o(1)\right)4^j<4^j
\]
for all sufficiently large $j$. Thus $q_j(c_j e_2)>-4^j$, and \cref{lem:untruncated-region} gives
\[
  u(c_j e_2)\ge-b_j-2^{1-j}\log 2.
\]
Since $u\le0$ and $\log c_j<0$, it follows that
\[
  0\le\frac{u(c_j e_2)}{\log c_j}\le \frac{b_j+2^{1-j}\log 2}{-\log c_j}\le 2^{1-j}+\frac{2^{1-j}\log 2}{-\log c_j}\longrightarrow 0.
\]
Finally, $c_j e_2\to0$. The liminf formula in \cref{lem:demailly-lelong} therefore gives $\nu(u,0)=0$.
\end{proof}

\begin{proof}[Proof of \cref{thm:main}]
The plurisubharmonic function $u$ is defined on $\mathbb D^2$ by \eqref{eq:u}. By \cref{prop:limit}, it is negative with an isolated pole at $0$, and $\e^u$ is continuous on $\mathbb D^2$. Finally, \cref{prop:lelong,prop:mass} give $\nu(u,0)=0$ and $\MA(u)=\delta_0$. Together, these conclusions prove the theorem.
\end{proof}

\subsection{Higher dimensions}

Let $n\ge3$, put $m=n-2$, and identify
\[
  \mathbb D^n=\mathbb D^2\times\mathbb D^m.
\]
Write the coordinates as $(\zeta,\eta)$ and let
\[
  v(\eta)=\log\|\eta\|_\infty,
  \qquad
  u_n(\zeta,\eta)=\max\{u(\zeta),v(\eta)\},
\]
where $u$ is the function in \cref{thm:main}.
Then this candidate $u_n$ will be the desired function. 

\begin{proof}[Proof of \cref{cor:higher-dimensional}]

First the function $u_n$ is negative and plurisubharmonic, and its pole set is $\{0\}$. Moreover,
\[
  \e^{u_n(\zeta,\eta)}=\max\{\e^{u(\zeta)},\|\eta\|_\infty\},
\]
so $\e^{u_n}$ is continuous.
Then we compute its Lelong number at the origin. Since $\nu(v,0)=1$, the max formula gives
\[
  \nu(u_n,0)=\min\{\nu(u,0),\nu(v,0)\}=0.
\]

It remains to compute the Monge--Amp\`ere measure. Fix $0<r<1$ and work on
\[
  \Omega_1=r\mathbb D^2,\qquad \Omega_2=r\mathbb D^m.
\]
The functions $u$ and $v$ are continuous and finite on their respective boundaries. We may therefore choose
\[
  R_1<\min_{\partial\Omega_1}u,\qquad R_2<\min_{\partial\Omega_2}v
\]
so that $\{u<R_1\}\Subset\Omega_1$ and $\{v<R_2\}\Subset\Omega_2$. For $T$ sufficiently large that $-T<\min\{R_1,R_2\}$, define the bounded nonnegative plurisubharmonic functions
\[
  u_T=\max\{u,-T\}+T \quad\text{on }\Omega_1,
  \qquad
  v_T=\max\{v,-T\}+T \quad\text{on }\Omega_2.
\]
Since $\MA(u)=\delta_0$ and $(\ddc v)^m=\delta_0$ by \cite[Chapter~III, Corollary~7.4]{DemaillyCADG}, the sweeping formula in \cref{lem:sweeping} shows that $\MA(u_T)$ and $(\ddc v_T)^m$ are probability measures supported, respectively, on
\[
  \{u=-T\}\subset\{u_T=0\}=\{u\le-T\},
  \qquad
  \{v=-T\}\subset\{v_T=0\}=\{v\le-T\}.
\]
In particular, neither measure charges the set on which its potential is positive. B{\l}ocki's max-product identity \cite[Theorem~7]{Blo00}, which is also the identity used in the proof of \cite[Theorem~5.7]{Wik05}, therefore gives
\begin{equation}\label{eq:truncated-product}
  \bigl(\ddc\max\{u_T(\zeta),v_T(\eta)\}\bigr)^n
  =(\ddc u_T)^2\wedge(\ddc v_T)^m
  \quad\text{on }\Omega_1\times\Omega_2.
\end{equation}
Here the right-hand side is the product of the two pulled-back measures. On the other hand,
\[
  \max\{u_T(\zeta),v_T(\eta)\}
  =\max\{u_n(\zeta,\eta),-T\}+T.
\]
The right-hand side of \eqref{eq:truncated-product} is thus a probability measure, and its support is contained in
\[
  \{u\le-T\}\times\{v\le-T\}.
\]
Since $\e^u$ and $\e^v=\|\eta\|_\infty$ are continuous and vanish only at their respective origins, the supports of these product measures shrink to the origin as $T\to\infty$; hence the measures converge weakly to $\delta_0$. The functions $\max\{u_n,-T\}$ decrease to $u_n$, whose unbounded locus is the single point $0$. Monotone continuity for the Monge--Amp\`ere operator at isolated poles \cite[Chapter~III, Corollary~4.3]{DemaillyCADG} now yields
\[
  \MA(u_n)=\delta_0
  \quad\text{on }r\mathbb D^n.
\]
Since $r<1$ was arbitrary, the same identity holds on $\mathbb D^n$. This completes the proof.
\end{proof}

Finally, we remark that in higher dimension a counterexample can also be constructed via a recursion process similar to that in dimension $2$.

\printbibliography

Long Li, \textsc{Institute of Mathematical Sciences, ShanghaiTech University}\par\nopagebreak
  \textit{Email address:} \texttt{lilong1@shanghaitech.edu.cn}

Mingchen Xia, \textsc{Institute of Geometry and Physics, University of Science and Technology of China}\par\nopagebreak
  \textit{Email address:} \texttt{xiamingchen2008@gmail.com}

\end{document}